\documentclass[review]{elsarticle}

\usepackage{lineno,hyperref}
\usepackage{graphicx}
\usepackage{xcolor,tikz}
\usepackage{float}
\usepackage{comment}

\usepackage{amsthm}
\usepackage{amsmath}
\usepackage{amsfonts}
\usepackage{enumerate}

\usepackage{xcolor,tikz}
\usepackage{float}

\newtheorem{teo}{Theorem}
\newtheorem{lem}{Lemma}
\newtheorem{cor}{Corollary}
\newtheorem{prop}{Proposition}

\newtheorem{defi}{Definition}

\newtheorem{conjecture}{Conjecture}

\theoremstyle{definition}
\newtheorem{rem}{Remark}

\DeclareMathOperator{\cof}{cof}

\modulolinenumbers[5]

\journal{Journal of \LaTeX\ Templates}

\begin{document}

\begin{frontmatter}

\title{{The determinant of the distance matrix of graphs with at most two cycles}}

\author[Sarmiento]{Ezequiel Dratman\fnref{anp}}
\ead{edratman@campus.ungs.edu.ar}

\author[Sarmiento]{Luciano N. Grippo\fnref{anp}}
\ead{lgrippo@campus.ungs.edu.ar}

\author[UNS]{Mart\'in {D.}~Safe\fnref{anp,uns,martin}}
\ead{msafe@uns.edu.ar}

\author[CEFET]{Celso~M.~da Silva Jr.}
\ead{celso.silva@cefet-rj.br}

\author[UFF]{Renata R.~Del-Vecchio\fnref{renata}}
\ead{rrdelvecchio@id.uff.br}

\address[Sarmiento]{CONICET - Instituto de Ciencias, Universidad Nacional de General Sarmiento, Los Polvorines, Argentina} 
  
\address[UNS]{{Departamento de Matem\'atica, Universidad Nacional del Sur (UNS), Bah\'ia Blanca, Argentina} and {INMABB, Universidad Nacional del Sur (UNS)-CONICET, Bah\'ia Blanca, Argentina}}
   
\address[CEFET]{DEMET and PPPRO, Centro Federal de Educa\c{c}\~ao Tecnol\' ogica Celso Suckow da Fonseca, Rio de Janeiro, Brazil} 

\address[UFF]{Departamento de An\'alise, Universidade Federal Fluminense, Niter\'oi, Brazil}

\fntext[anp]{{Work partially supported by ANPCyT PICT 2017-1315}}
\fntext[uns]{{Work partially supported by UNS {Grants} PGI 24/L103 {and PGI 24/L115}.}}
\fntext[martin]{Work partially supported {by} MATH-AmSud 18-MATH-01.}
\fntext[renata]{Work partially supported by CNPq Grants 408494/2016-6 and 305411/2016-0.}


\begin{abstract} 
Let $G$ be a connected graph on $n$ vertices and $D(G)$ its distance matrix. The formula for computing the determinant of this matrix in terms of the number of vertices is known when the graph is either a tree or {a} unicyclic graph.  
In this work we generalize these results, obtaining the determinant of the distance matrix for {all graphs} in a {class, including trees, unicyclic and bicyclic graphs. This class actually includes graphs with many cycles, provided that each block of the graph is at most bicyclic.}
\end{abstract}


\end{frontmatter}

\linenumbers

\section{Introduction}\label{intro}

A \emph{graph} $G=(V,E)$ consists of a set $V$ of \emph{vertices} and a set $E$ of \emph{edges}. We will consider graphs without multiple edges and without loops. Let $G$
be a connected graph on $n$ vertices with vertex set $V=\{v_1,\ldots,v_n\}$. The \emph{distance between vertices} $v_i$ and $v_j$, denoted $d(v_i,v_j)$, is the number of edges of {a} shortest path from $v_i$ to $v_j$. The \emph{distance matrix} of $G$, denoted $D(G)$, is the $n\times n$ symmetric matrix having its $(i,j)$-entry equal to $d(v_i,v_j)$. We also use $d_{i,j}$ to denote $d(v_i,v_j)$.

The distance matrix has been widely studied in the literature. The  interest in this matrix was motivated by  the connection with a communication problem  (see~\cite{graham-pollack-1971,graham-lovasz-1978} for more {details}). 
In an early article,~\cite{graham-pollack-1971}, Graham and Pollack presented a remarkable result, proving that the determinant of the distance matrix of a tree $T$ on $n$ vertices only depends on $n,$ being equal to ${(-1)^{n-1}(n-1)2^{n-2}}$. 
This result was generalized by Graham, Hoffman, and
Hosoya in 1977 \cite{graham-hoffman-hosoya-1977}, who proved that, for any graph $G$, the determinant of $D(G)$ depends only on the blocks of $G$.


In 2005, more than 30 years after the result of Graham and Pollack on trees, Bapat, Kirkland and Neumann~\cite{bapat-kirkland-neumann-2005} exhibited a formula for the determinant of the distance matrix of a unicyclic graph. Specifically, they proved that the determinant is zero when its only cycle has 
an even number of edges, {whereas} if the graph has $2k+1+m$ vertices and a cycle with $2k+1$ edges, the determinant is equal to $(-2)^m\left[k(k+1)+\frac{2k+1}{2}m\right]$.

{For a {bicyclic} graph,  the determinant can be easily computed in the case where the cycles have no common edges, since its blocks are  edges {and} cycles. In a conference article \cite{nosso}, we presented some advances for the remaining cases; i.e., when the cycles share at least one edge. Besides, we conjectured the formula for the remaining cases. {In} the present article, we completely solve these conjectures, extending the formula of the determinant of $D(G)$ to graphs $G$ having bicyclic blocks as well as trees and unicyclic blocks.}

This paper is organized as follows. In {Section}~\ref{sec: preliminaries} we present some {basic notations, preliminary results, and} we briefly describe previous results in connection with the determinant of the distance matrix of a bicyclic graph.  In {Sections}~\ref{sec: theta}  we consider the determinant of the distance matrix of a $\theta$-graph, a $\theta$-graph plus a pendant vertex and a $\theta$-graph attached to a path, where the definition of a $\theta$-graph is stated in {Section}~\ref{sec: preliminaries}. In the last theorem, we present a formula for the determinant of a graph arised from a tree by the addition of at most two edges (graphs at most bicyclic). 

\section{Definitions and preliminary results}\label{sec: preliminaries}

A \emph{tree} is a connected acyclic graph. A \emph{unicyclic graph} is a connected graph with as many edges as vertices.
The path and the cycle on $n$ vertices are denoted by $C_n$ and $P_n$, respectively. 

The determinant and the cofactor of the distance matrix of a cycle are known and they are given  in the lemma below. We remember that the cofactor  for any square matrix $A$, denoted by $\cof (A),$ is the sum of the cofactors of $A.$

\begin{lem}[\cite{bapat-kirkland-neumann-2005, cof}] \label{1}
For each $n \geq 3$:
\begin{itemize}
\item if $n$ is odd, $\det D(C_n) = (n^2 - 1)/4$ and $\cof D(C_n) = n;$
\item if $n$ is even, $\det D(C_n) = 0$ and $\cof D(C_n) = 0.$
\end{itemize}
\end{lem}

In \cite{bapat-kirkland-neumann-2005} the {determinant} of $D(G)$ was {obtained} when $G$ is a unicyclic graph.

\begin{teo}[\cite{bapat-kirkland-neumann-2005}]
Let $G$ be a unicyclic graph consisting of a cycle of length $l$ plus $m$ edges outside the cycle. If $l$ is even, then $\det D(G) = 0;$ otherwise:
$$\det D(G) = (-2)^m\frac{l^2 + 2ml - 1}{4}$$
\end{teo}

 A \emph{cut-vertex} of a connected graph is a vertex whose removal disconnects the graph. A \emph{block of a graph} G is a maximal connected subgraph of G
having no cut-vertices.
A \emph{block} is a connected graph having no cut-vertices.

 In \cite{graham-hoffman-hosoya-1977} it was proved that if the blocks of
a graph $G$ are $G_1,G_2,\hdots,G_k,$ then $\det D(G)$ depends only on the $\det D(G_1), \det D(G_2 ),\hdots, \det D(G_k)$ and
$\cof D(G_1), \cof D(G_2 ),\hdots, \cof D(G_k).$

\begin{teo}[\cite{graham-hoffman-hosoya-1977}] \label{2}
If G is a connected graph whose blocks are $G_1, G_2 , \hdots , G_k,$  then 
$$ \det D(G)= \sum\limits_{i=1}^k \det D(G_i) \prod \limits_{j \in \{ 1,2,\hdots,k\}-\{i\}}\cof D(G_j)$$
and
$$\cof D(G) = \prod\limits_{i=1}^{k}\cof D(G_i).$$
\end{teo}

A {\emph{cactus}} is a graph in which {each} two cycles have at most one vertex in common. By definition, every unicyclic graph is a cactus. Moreover, each block of a cactus on at least two vertices is either an edge or a
cycle. As $\det D(G)$ depends only on the blocks of $G$ and $\det D$ and $\cof D$ are known for an edge and for the cycles, we obtain the next corollary as an immediate consequence of Lemma \ref{1} and Theorem \ref{2}.

\begin{cor}
Let $G$ be a connected cactus having precisely $c$ cycles whose
lengths are $l_1,l_2,\hdots,l_c$ plus $m$ other edges outside these cycles.
\begin{itemize}
\item If some of $l_1,l_2,\hdots,l_c$ is even, then $\det D(G) = 0.$
\item Otherwise (i.e., if all of $l_1,l_2,\hdots,l_c$ are odd),
$$ \det D(G)=(-2)^m\left(\prod\limits_{i=1}^{c}l_i\right)\left(\frac{m}{2}+\sum\limits_{i=1}^{c}\frac{l_i^2-1}{4l_i}\right).$$
\end{itemize}
\end{cor}

A \emph{bicyclic graph} is a graph obtained by adding an edge to a unicyclic graph. The special case of $c = 2$ in  the  formula  of the above corollary was also obtained in \cite{GoZhXu2013} by alternative means, corresponding to a special class of bicyclic graphs.

As $\det D$ for all cacti is known, in order to find $\det D$ for all bicyclic
graphs, it is enough to find $\det D $ and $\cof D$ for bicyclic blocks.

\begin{defi}
Let $P_{l+1},P_{p+1}, P_{q+1}$ be three vertex disjoint paths, $l \geq 1$ and $p,q \geq 2,$ each of them having endpoints, $v^l_1, v^l_2, v^p_1, v^p_2, v^q_1, v^q_2$, respectively.  We denote by $\theta(l,p,q)$-graph, or simply $\theta$-graph,  the graph obtained by identifying the vertices $v^l_1, v^p_1, v^q_1 $ as one vertex, and proceeding in the same way for $v^l_2, v^p_2, v^q_2$.
\end{defi}



{Note that $\theta(l,p,q)$-graph is a bicyclic graph, with no pendant edge, whose cycles share at least one edge. In \cite{nosso}, we proved the following results:}

\begin{prop}[{\cite[Lemma 3.1]{nosso}}]
For every positive integer $k$, $$\det D(\theta(2,2, 2k+1)) = 4(k^2+k-1).$$ 
\end{prop}

\begin{prop}[{\cite[Lemma 3.2]{nosso}}]\label{lem: theta}
Let $G$ be one of the graphs bellow: 
\begin{itemize}
 \item $\theta(1,2k-1,2k-1),$ for $k \ge 2;$
 \item $\theta(2,2,2k-2),$ for $k\ge3;$
 \item $\theta(l,p,q),$ for {$l \geq 2,$ $p\geq 3$, and $q\geq 3$}.
\end{itemize}
Then, $\det{{D}}(G)= 0.$ 
\end{prop}

\section{Bicyclic graphs}\label{sec: theta}

The next theorem gives the determinant of $D(G)$ when $G=\theta(l,p,q)$, completing the remaining cases in \cite{nosso}.


\begin{teo} \label{theta}
The following assertions hold:
\begin{enumerate}[(a)]

\item If $G = \theta(1, p, q)$ for even integers p and q, then
$\det D(G) =\frac{-(p + q)^2}{4}.$
\item If $G = \theta(2, 2, 2)$, then $\det D(G) = -16.$
\item If $G = \theta(2, 2, q)$ for some odd integer $q > 1$, then
$\det D(G) = q^2 -5.$
\item Otherwise, $\det D(G) = 0.$

\end{enumerate}
\end{teo}
\begin{proof}

Items (c) and (d) have been proven in \cite{nosso} and {correspond} to Proposition \ref{1} and Proposition \ref{2} respectively. Case (b) can be easily computed. The proof of case (a) will be  divided in the following 2 cases:

\textbf{Case 1:} 

Let $G=\theta(1,2,2k),$  for some $k \ge 1$, with its vertices labeled as in Figure~\ref{Figura theta(1,2,2k)}.
\begin{figure}[H]
  \centering
  \begin{tikzpicture}[scale=1]
      
      \fill (0,0) circle(2pt) node[left] {$v_{1}$};
      \draw[-, thick] (0,0)--(1,-0.5); 
      \draw[-, thick] (0,0)--(1,0.5); 
      
      \draw[-, thick] (1,-0.5)--(1,0.5); 
      
      \fill (1,0.5) circle(2pt) node[above] {$v_2$}; 
      \draw[-, thick] (1,0.5)--(1.5,0.5);
      \draw[-, dotted, thick] (1.5,0.5)--(2.5,0.5);
      \draw[-, thick] (2.5,0.5)--(3,0.5);

      \fill (3,0.5) circle(2pt) node[above] {$v_{k+1}$};

      \fill (4,0) circle(2pt) node[right] {$v_{k+2}$};
      \draw[-, thick] (4,0)--(3,-0.5); 
      \draw[-, thick] (4,0)--(3,0.5); 
           
      \fill (3,-0.5) circle(2pt) node[below] {$v_{k+3}$};
      \draw[-, thick] (3,-0.5)--(2.5,-0.5);
      \draw[-, dotted, thick] (2.5,-0.5)--(1.5,-0.5);
      \draw[-, thick] (1.5,-0.5)--(1,-0.5);

      \fill (1,-0.5) circle(2pt) node[below] {$v_{2k+2}$}; 
      
   \end{tikzpicture}
   \caption{$\theta(1,2,2k)$} \label{Figura theta(1,2,2k)}  
\end{figure}
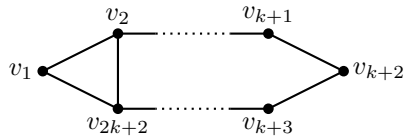  
The distance matrix of $\theta(1,2,2k)$ is
$$
D(\theta(1,2,2k)) =  
\left(\begin{array}{cc}
 0  & v^t\\
 v & D(C_{2k+1})
\end{array}\right),
$$
where $D(C_{2k+1})$ is the distance matrix of the cycle induced by the vertices $v_2, \dots, v_{2k+2}$ and $v^t = 
(1,2,\dots,k,k+1,k,\dots,2,1)$. 

From \cite{bapat-kirkland-neumann-2005}, we know that 
\begin{equation}\label{eq inv C_{2k+1}}
D(C_{2k+1})^{-1} = -2 I - C^k - C^{k+1} + \frac{2k+1}{k(k+1)}J,
\end{equation}
and $\det D(C_{2k+1}) = k(k+1)$, {where $J$ is the all ones matrix, with appropriate size,} and  $C$ is the cyclic permutation matrix of order $2k + 1$
having $C_{i,i+1} = 1$ for $i = 1,\dots,2k + 1$, taking indices modulo $2k + 1$. {Therefore, we have that} 
\begin{equation}\label{eq inv theta(1,2,2k)}
D(\theta(1,2,2k))^{-1} = M_1^t M_2 M_1, 
\end{equation}
where
\begin{eqnarray*}
M_1 & = &
\left(\begin{array}{cc}
 1  & -v^t D(C_{2k+1})^{-1} \\
 0 & I
\end{array}\right), 
\\[2ex]
M_2 &=&  
\left(\begin{array}{cc}
 (-v^t D(C_{2k+1})^{-1}v)^{-1}  & 0 \\
 0 & D(C_{2k+1})^{-1}
\end{array}\right),
\end{eqnarray*}

and
\begin{eqnarray}
\nonumber \det D(\theta(1,2,2k)) &=& \det M_2^{-1} = - v^t D(C_{2k+1})^{-1}v \ \det D(C_{2k+1})\\
\label{eq det theta(1,2,2k)}                                        &=& - v^t D(C_{2k+1})^{-1}v \ k(k+1).
\end{eqnarray}
Now we will calculate $v^t D(C_{2k+1})^{-1}v$, {using equation} \eqref{eq inv C_{2k+1}} we obtain
{\footnotesize 
\begin{eqnarray}
\nonumber v^tD(C_{2k+1})^{-1} v & = &  -2 v^tv - v^t C^k v - v^t C^{k+1} v + \frac{2k+1}{k(k+1)} v^t J v\\
\nonumber                      & = &  -4 \sum_{i=1}^k i^2 - 2 (k+1)^2  - 2 \sum_{i=1}^k i (k+1-i) \\
\label{eq vt C(2k+1)^{-1} v}                      &   &   - 2 \sum_{i=1}^{k+1} i (k+2-i) + \frac{2k+1}{k(k+1)} (k+1)^4\\
\nonumber                      & = &  - 2 \sum_{i=1}^k i (k+1) - 2 \sum_{i=1}^{k+1} i (k+2) + \frac{(2k+1)(k+1)^3}{k}\\
\nonumber                      & = &  - k (k+1)^2 -  (k+1)(k+2)^2 + \frac{(2k+1)(k+1)^3}{k} = \frac{k+1}{k}.
 \end{eqnarray}
} 
Combining this result with \eqref{eq det theta(1,2,2k)}, we deduce that  
\begin{equation}\label{eq det theta(1,2,2k) final}
\det D(\theta(1,2,2k)) = -(k+1)^2 = - \frac{(2k+2)^2}{4} = -n^2 (-2)^{-2},
\end{equation}
with $n=p+q$, {where $p=2$ and $q=2k$}.

\textbf{Case 2:} 

 Let  $H=\theta(1,2s,2k)$ and $G=\theta(1,2(s-1),2(k+1))$, for some $k \ge {2}$ and $s \ge 2,$ with  its vertices labeled as in Figure~\ref{Figura theta(1,2s,2k)} and Figure~\ref{Figura 
theta(1,2(s-1),2(k+1))} respectively.

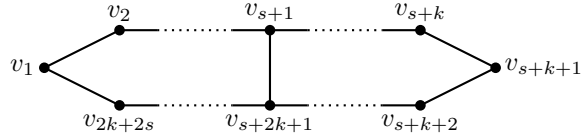
\begin{figure}[H]
  \centering
  \begin{tikzpicture}[scale=1]
      
      \fill (-3,0) circle(2pt) node[left] {$v_{1}$};
      \draw[-, thick] (-3,0)--(-2,-0.5); 
      \draw[-, thick] (-3,0)--(-2,0.5); 
      
      \fill (-2,0.5) circle(2pt) node[above] {$v_{2}$}; 
      \draw[-, thick] (-2,0.5)--(-1.5,0.5);
      \draw[-, dotted, thick] (-1.5,0.5)--(-0.5,0.5);
      \draw[-, thick] (-0.5,0.5)--(0,0.5);
      
      \draw[-, thick] (0,-0.5)--(0,0.5); 
      
      \fill (0,0.5) circle(2pt) node[above] {$v_{s+1}$}; 
      \draw[-, thick] (0,0.5)--(0.5,0.5);
      \draw[-, dotted, thick] (0.5,0.5)--(1.5,0.5);
      \draw[-, thick] (1.5,0.5)--(2,0.5);

      \fill (2,0.5) circle(2pt) node[above] {$v_{s+k}$};

      \fill (3,0) circle(2pt) node[right] {$v_{s+k+1}$};
      \draw[-, thick] (3,0)--(2,-0.5); 
      \draw[-, thick] (3,0)--(2,0.5); 
           
      \fill (2,-0.5) circle(2pt) node[below] {$v_{s+k+2}$};
      \draw[-, thick] (2,-0.5)--(1.5,-0.5);
      \draw[-, dotted, thick] (1.5,-0.5)--(0.5,-0.5);
      \draw[-, thick] (0.5,-0.5)--(0,-0.5);

      \fill (0,-0.5) circle(2pt) node[below] {$v_{s+2k+1}$}; 
      
      \fill (-2,-0.5) circle(2pt) node[below] {$v_{2k+2s}$}; 
      \draw[-, thick] (-2,-0.5)--(-1.5,-0.5);
      \draw[-, dotted, thick] (-1.5,-0.5)--(-0.5,-0.5);
      \draw[-, thick] (-0.5,-0.5)--(0,-0.5);
            
   \end{tikzpicture}
   \caption{$\theta(1,2s,2k)$} \label{Figura theta(1,2s,2k)}  
\end{figure}

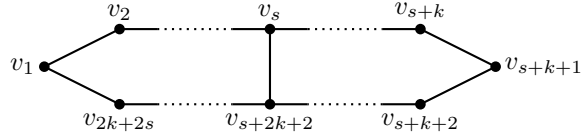
\begin{figure}[H]
  \centering
  \begin{tikzpicture}[scale=1]
      
      \fill (-3,0) circle(2pt) node[left] {$v_{1}$};
      \draw[-, thick] (-3,0)--(-2,-0.5); 
      \draw[-, thick] (-3,0)--(-2,0.5); 
      
      \fill (-2,0.5) circle(2pt) node[above] {$v_{2}$}; 
      \draw[-, thick] (-2,0.5)--(-1.5,0.5);
      \draw[-, dotted, thick] (-1.5,0.5)--(-0.5,0.5);
      \draw[-, thick] (-0.5,0.5)--(0,0.5);
      
      \draw[-, thick] (0,-0.5)--(0,0.5); 
      
      \fill (0,0.5) circle(2pt) node[above] {$v_{s}$}; 
      \draw[-, thick] (0,0.5)--(0.5,0.5);
      \draw[-, dotted, thick] (0.5,0.5)--(1.5,0.5);
      \draw[-, thick] (1.5,0.5)--(2,0.5);

      \fill (2,0.5) circle(2pt) node[above] {$v_{s+k}$};

      \fill (3,0) circle(2pt) node[right] {$v_{s+k+1}$};
      \draw[-, thick] (3,0)--(2,-0.5); 
      \draw[-, thick] (3,0)--(2,0.5); 
           
      \fill (2,-0.5) circle(2pt) node[below] {$v_{s+k+2}$};
      \draw[-, thick] (2,-0.5)--(1.5,-0.5);
      \draw[-, dotted, thick] (1.5,-0.5)--(0.5,-0.5);
      \draw[-, thick] (0.5,-0.5)--(0,-0.5);

      \fill (0,-0.5) circle(2pt) node[below] {$v_{s+2k+2}$}; 
      
      \fill (-2,-0.5) circle(2pt) node[below] {$v_{2k+2s}$}; 
      \draw[-, thick] (-2,-0.5)--(-1.5,-0.5);
      \draw[-, dotted, thick] (-1.5,-0.5)--(-0.5,-0.5);
      \draw[-, thick] (-0.5,-0.5)--(0,-0.5);
      
  \end{tikzpicture}
  \caption{$\theta(1,2(s-1),2(k+1))$} \label{Figura theta(1,2(s-1),2(k+1))}  
\end{figure}
The distance matrices of $G$ and $H$ are
$$
D(G) = 
\left(
\begin{array}{cc}
P & A^t\\
A & P
\end{array}
\right)\quad\text{and}\quad
D(H) = 
\left(
\begin{array}{cc}
P & B^t\\
B & P
\end{array}
\right),
$$
where 
{\footnotesize 
\begin{eqnarray} \label{fla1}
P &=& \sum_{i=1}^{k+s} \sum_{j=1}^{k+s} |i-j| \ e_ie_j^t,  
\end{eqnarray}
}%
is the distance matrix of $P_{k+s}$ (the path on $k+s$ vertices), 
and $e_i$ denotes a vector having an entry equal to $1$ on the $i$-th coordinate and $0$'s {in} the remaining coordinates. Moreover,
{\footnotesize 
\begin{eqnarray}\label{fla2}
B^t &=& \sum_{j=1}^{k+s} (k+s+1-j) e_1 e_j^t + \sum_{i=2}^{k+s} (k+s+1-i) e_i e_1^t \nonumber \\
    & & + \sum_{i=2}^{s+1} \sum_{j=2}^{k+1} (s+k+3-j-i) e_i e_j^t + \sum_{i=s+2}^{s+k} \sum_{j=k+2}^{k+s} (j+i-s-k-1) e_i e_j^t \nonumber \\
    & & + \sum_{i=3}^{s+1} s \ e_i e_{i+k-1}^t + \sum_{i=2}^{s+1}\sum_{\substack{j=k+2\\j \neq i+k-1}}^{k+s} |r_{2s+1}(1-k+j-i) - s-1| e_i e_j^t \nonumber \\
    & & + \sum_{i=s+2}^{s+k}  k \ e_i e_{i-s}^t + \sum_{i=s+2}^{s+k}\sum_{\substack{j=2\\j \neq i-s}}^{k+1} |r_{2k+1}(s+j-i) - k-1| e_i e_j^t 
\end{eqnarray}
}
and
{\footnotesize 
\begin{eqnarray}\label{fla3}
A^t &=& \sum_{j=1}^{k+s} (k+s+1-j) e_1 e_j^t + \sum_{i=2}^{k+s} (k+s+1-i) e_i e_1^t \nonumber\\
    & & + \sum_{i=2}^{s} \sum_{j=2}^{k+2} (s+k+3-j-i) e_i e_j^t + \sum_{i=s+1}^{s+k} \sum_{j=k+3}^{k+s} (j+i-s-k-1) e_i e_j^t \nonumber\\
    & & + \sum_{i=3}^{s} (s-1) \ e_i e_{i+k}^t + \sum_{i=2}^{s}\sum_{\substack{j=k+3\\j \neq i+k}}^{k+s} |r_{2s-1}(j-k-i) - s| e_i e_j^t \nonumber \\
    & & + \sum_{i=s+1}^{s+k}  (k+1) \ e_i e_{i-s+1}^t + \sum_{i=s+1}^{s+k}\sum_{\substack{j=2\\j \neq i-s+1}}^{k+2} |r_{2k+3}(s+j-i-1) - k-2| e_i e_j^t, 
\end{eqnarray}
}
where $r_{\alpha}(\beta)$ represent the remainder when integer $\beta$ is divided 
by $\alpha$. 

It is easy to see that $P$ is invertible and 
{\footnotesize
\begin{eqnarray*}
P^{-1} &=& -\frac{k+s-2}{2(k+s-1)} \ e_1 e_1^t - \frac{k+s-2}{2(k+s-1)} \ e_{k+s} e_{k+s}^t-\sum_{i=2}^{k+s-1}  e_ie_i^t  \\
             & & + \sum_{i=1}^{k+s-1} \frac{1}{2} \ e_i e_{i+1}^t + \sum_{i=2}^{k+s} \frac{1}{2} \ e_{i} e_{i-1}^t \\
             & & + \frac{1}{2(k+s-1)} \ e_1 e_{k+s}^t + \frac{1}{2(k+s-1)} \ e_{k+s} e_{1}^t.
\end{eqnarray*}
}
We define 
$$
N := 
\left(
\begin{array}{cc}
    I        &     0           \\
    (A  - M B) P^{-1}  &   M         
\end{array}
\right),
$$
{where}
$$
M := e_1 e_1^t + e_2 e_{k+1}^t - e_2 e_{k+s}^t + \sum_{i=2}^{k+s} e_{i}e_{i-1}^t.  
$$
We claim that 
\begin{equation}\label{eq D(G) = N D(H) Nt}
D(G) = N \cdot D(H) \cdot N^t.
\end{equation}
Indeed, it is easy to see that 
$$
N
\cdot
D(H)
\cdot
N^t 
= 
\left(
\begin{array}{cc}
P & A^t\\
A & \widehat{P}
\end{array}
\right).
$$
where 
$$
\widehat{P} = A \ P^{-1} \ (A^t  - B^t M^t) + (A  - M B) \ P^{-1} \ B^t \ M^t + M \ P \ M^t.
$$
Hence, it is sufficient to prove that $\widehat{P} = P$.
We first compute $M \ P \ M^t$. Since 
$$
M^t = e_1 e_1^t + e_{k+1}e_2^t - e_{k+s}e_2^t + \sum_{i=2}^{k+s} e_{i-1}e_{i}^t,  
$$
we have
{\footnotesize
$$
M \ P = \sum_{j=1}^{k+s} (j-1) \ e_1e_j^t + \sum_{j=1}^{k+s} (|k+1-j| + 2j - 1- k - s) \ e_2e_j^t + \sum_{i=3}^{k+s}\sum_{j=1}^{k+s} |i-1-j| \ 
e_{i}e_j^t
$$
}
and
{\footnotesize
\begin{eqnarray}
\label{eq: M P Mt} M \ P \ M^t &=& (1 - s) \ e_2e_1^t + (1-s) \ e_1e_2^t + 4 (1 - s) \ e_2e_2^t \\ 
\nonumber                  & & + \sum_{i=3}^{k+s} (i-2) \ e_{i}e_1^t + \sum_{j=3}^{k+s} (j-2) \ e_1e_{j}^t\\ 
\nonumber                  & & + \sum_{j=3}^{k+s} (|k+2-j| + 2j - 3- k - s) \ e_2e_{j}^t \\ 
\nonumber                  & & + \sum_{i=3}^{k+s} (|k+2-i| + 2i - 3- k - s) \ e_{i}e_2^t \\
\nonumber                  & & + \sum_{i=3}^{k+s} \sum_{j=3}^{k+s} |i-j| \ e_{i}e_{j}^t.
\end{eqnarray}
}
We continue obtaining $A^t - B^t M^t$, multiplying $B^t$ with $M^t$ we have
{\footnotesize 
\begin{eqnarray*}
B^t M^t&=& B^t e_1 e_1^t + B^t e_{k+1}e_2^t - B^t e_{k+s}e_2^t + \sum_{i=2}^{k+s} B^t e_{i-1}e_{i}^t\\
&=& \sum_{i=1}^{k+s} (k+s+1-i) e_i e_1^t  + (k+2s-1) e_1 e_2^t  +  \sum_{j=3}^{k+s} (k+s+2-j) e_1 e_{j}^t \\
&&   + \sum_{i=2}^{s} (k+2s+3-3i) e_i e_2^t + \sum_{i=s+1}^{k+s} (k+2-i) e_i e_2^t \\
    & & + \sum_{i=2}^{s+1} \sum_{j=3}^{k+2} (s+k+4-j-i) e_i e_{j}^t + \sum_{i=s+2}^{s+k} \sum_{j=k+3}^{k+s} (j+i-s-k-2) e_i e_{j}^t\\
    & & + \sum_{i=3}^{s}  s \ e_i e_{i+k}^t + \sum_{i=2}^{s+1}\sum_{\substack{j=k+3\\j \neq i+k}}^{k+s} |r_{2s+1}(j-i-k) - s-1| e_i e_{j}^t 
\\
    & & + \sum_{i=s+2}^{s+k}  k \ e_i e_{i-s+1}^t + \sum_{i=s+2}^{s+k}\sum_{\substack{j=3\\j \neq i-s+1}}^{k+2} |r_{2k+1}(s+j-i-1) - k-1| e_i e_{j}^t. 
\end{eqnarray*}
}
From this, we deduce that
{\footnotesize 
\begin{eqnarray*}
A^t - B^tM^t &=& \sum_{i=1}^{s} (2i-2-s) e_i e_2^t  + \sum_{i=s+1}^{s+k} s \ e_i e_2^t + \sum_{i=s+1}^{s+k} \sum_{j=3}^{k+s} e_i e_j^t - \sum_{i=1}^{s} 
\sum_{j=3}^{k+s}  e_i e_j^t. 
\end{eqnarray*}
}
It follows that 
{\footnotesize 
\begin{eqnarray*}
P^{-1}(A^t - B^tM^t) &=& \Big(-\frac{k+s-2}{2(k+s-1)} \ e_1 e_1^t + \frac{1}{2} \ e_1 e_{2}^t + \frac{1}{2(k+s-1)} \ e_1 e_{k+s}^t  \\
             & & + \sum_{i=2}^{k+s-1} \frac{1}{2} \ e_i e_{i+1}^t -\sum_{i=2}^{k+s-1}  e_ie_i^t  + \sum_{i=2}^{k+s-1} \frac{1}{2} \ e_{i} e_{i-1}^t \\
             & &  + \frac{1}{2(k+s-1)} \ e_{k+s} e_{1}^t + \frac{1}{2} \ e_{k+s} e_{k+s-1}^t- \frac{k+s-2}{2(k+s-1)} \ e_{k+s} e_{k+s}^t \Big)\\
             & & \cdot \Big(\sum_{i=1}^{s} (2i-2-s) e_i e_2^t  + \sum_{i=s+1}^{s+k} s \ e_i e_2^t + \sum_{i=s+1}^{s+k} \sum_{j=3}^{k+s} e_i e_j^t - 
                  \sum_{i=1}^{s} \sum_{j=3}^{k+s}  e_i e_j^t \Big)\\ 
             &=& e_1e_2^t + \sum_{i=3}^{k+s} e_{s} e_{i}^t - \sum_{i=2}^{k+s} e_{s+1} e_{i}^t. 
\end{eqnarray*}
}
Finally, we see that 
{\footnotesize 
\begin{eqnarray*}
A P^{-1}(A^t - B^tM^t) &=& A \Big(e_1e_2^t + \sum_{i=3}^{k+s} e_{s} e_{i}^t - \sum_{i=2}^{k+s} e_{s+1} e_{i}^t\Big)\\
                             &=& s \ e_{1} e_{2}^t + (s-2) e_{2} e_{2}^t \\
                              & & + \sum_{i=3}^{k+1} (s-3) e_{i} e_{2}^t + \sum_{i=k+2}^{k+s} (s+2k+1-2i) e_{i} e_{2}^t \\
                              & & + \sum_{j=3}^{k+s} e_{1} e_{j}^t - \sum_{i=3}^{k+s} \sum_{j=3}^{k+s} e_{i} e_{j}^t,  
\end{eqnarray*}
}
and
{\footnotesize 
\begin{eqnarray*}
(A - M B)P^{-1} B^t M^t &=&  \Big(e_2 e_1^t + \sum_{i=3}^{k+s} e_{i} e_{s}^t - \sum_{i=2}^{k+s} e_{i} e_{s+1}^t \Big) B^t M^t\\ 
                              &=& s \ e_{2} e_{1}^t + (3s-2) e_{2} e_{2}^t \\
                              & & + \sum_{j=3}^{k+1} (s-1) e_{2} e_{j}^t + \sum_{j=k+2}^{k+s} (s+2k+3-2j) e_{2} e_{j}^t \\
                              & &  + \sum_{i=3}^{k+s} e_{i} e_{1}^t + \sum_{i=3}^{k+s} 2 e_{i} e_{2}^t + \sum_{i=3}^{k+s} \sum_{j=3}^{k+s} e_{i} 
e_{j}^t. 
\end{eqnarray*}
}  
Thus
{\footnotesize 
\begin{eqnarray*}
A P^{-1}(A^t - B^tM^t) + (A - M B)P^{-1} B^t M^t  &=& s \ e_{1} e_{2}^t + s \ e_{2} e_{1}^t + 4(s-1) e_{2} e_{2}^t \\
                              & & + \sum_{j=3}^{k+s} e_{1} e_{j}^t + \sum_{i=3}^{k+s} e_{i} e_{1}^t\\  
                              & & + \sum_{i=3}^{k+s} (s + k + 1 - i -|k+2-i|) e_{i} e_{2}^t \\
                              & & + \sum_{j=3}^{k+s} (s + k + 1 - j -|k+2-j|) e_{2} e_{j}^t. 
\end{eqnarray*}
}  
Therefore, by \eqref{eq: M P Mt}, we obtain
$$
\widehat{P} = A P^{-1}(A^t - B^tM^t) + (A - M B)P^{-1} B^t M^t + M {P} M^t = P.  
$$ 
This complete the proof of \eqref{eq D(G) = N D(H) Nt}. 
Futhermore, since $\det N \cdot \det N^t = 1$, we deduce that 
$$
\det D(G) =  \det D(H). 
$$
Combining this with \eqref{eq det theta(1,2,2k) final}, using an inductive argument, we have 
$$
\det{D}(\theta(1,2s,2k)) = -n^2(-2)^{-2},
$$
with $n=p+q$, where $p=2s$ and $q=2k$.
\end{proof}

{In order to compute} {$\cof D(\theta(l,p,q))$,} {we need firstly compute the determinant of the} {graphs defined as follows}.

\begin{defi}
For each positive integers $l, p, q$ such that at most one of them is
1, we denote by $\theta^{\prime}(l, p, q)$ {any} graph that arises from $\theta(l, p, q)$ by
adding a pendant edge (see {Figures}~\ref{Figura F(1,2,2k,1)} {and~\ref{Figura F(1,2s,2k,1)}}).
\end{defi}

{Since every graph $\theta'(l,p,q)$ has as blocks $\theta(l,p,q)$ and one edge, if follows from Theorem~\ref{2} that $\det D(\theta'(l,p,q))$ and $\cof D(\theta'(l,p,q))$ are well defined (i.e., they are independent of the vertex of $\theta(l,p,q)$ to which the pendant edge is attached in order to obtain $\theta'(l,p,q)$). Moreover,} {from Theorem \ref{2}, it follows that}
\begin{equation*}{\cof D(\theta({l,p,q})) = -2 \det D(\theta({l,p,q})) - \det D(\theta^{\prime}({l,p,q})).}\end{equation*}

\begin{teo} \label{theta prime}
Let $G= \theta^{\prime}(l, p, q)$, for integers $l,p,q$ such that at most one of them is
1. Then, the following assertions hold:
\begin{enumerate}[(a)]
\item $G = \theta^{\prime}(1, p, q)$ for some even integers $p$ and $q$, then $\det D(G)= \frac{(1 + p + q)^2 - 1}{2}$
\item If $G = \theta^{\prime}(2, 2, 2)$, then
$\det D(G) = -16.$
\item If $G = \theta^{\prime}(2, 2, q)$ for some odd integer $q > 1$, then $\det D(G) = -2(q^2 + 2q - 9).$
\item Otherwise, $\det D(G) = 0.$
\end{enumerate}
\end{teo}

\begin{proof}

Once more,  items (c) and (d) have already been proven in \cite{nosso} and (b) can be computed directly. The proof of case $(a)$ will be  divided in the following 2 cases. {All along this proof, $\theta'(l,p,q)$ denotes the graph that arises from $\theta(l,p,q)$ by adding a pendant edge incident precisely to the midpoint of the path of length $p$ joining the two vertices of degree $3$ of $\theta(l,p,q)$. {Notice that in Figures~\ref{Figura F(1,2,2k,1)} and~\ref{Figura F(1,2s,2k,1)} such midpoint is the vertex $v_1$.}}

\textbf{Case 1:}  

Let $G=\theta^{\prime}(1,2,2k),$ for some $k \ge 1,$ with its vertices labeled as in Figure~\ref{Figura F(1,2,2k,1)}.
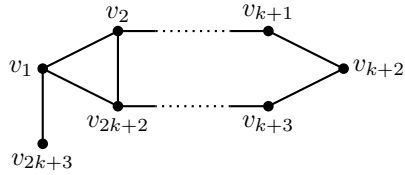
\begin{figure}[H]
  \centering
  \begin{tikzpicture}[scale=1]
      
      \fill (0,0) circle(2pt) node[left] {$v_{1}$};
      \draw[-, thick] (0,0)--(1,-0.5); 
      \draw[-, thick] (0,0)--(1,0.5); 
      
      \draw[-, thick] (1,-0.5)--(1,0.5); 
      
      \fill (1,0.5) circle(2pt) node[above] {$v_2$}; 
      \draw[-, thick] (1,0.5)--(1.5,0.5);
      \draw[-, dotted, thick] (1.5,0.5)--(2.5,0.5);
      \draw[-, thick] (2.5,0.5)--(3,0.5);

      \fill (3,0.5) circle(2pt) node[above] {$v_{k+1}$};

      \fill (4,0) circle(2pt) node[right] {$v_{k+2}$};
      \draw[-, thick] (4,0)--(3,-0.5); 
      \draw[-, thick] (4,0)--(3,0.5); 
           
      \fill (3,-0.5) circle(2pt) node[below] {$v_{k+3}$};
      \draw[-, thick] (3,-0.5)--(2.5,-0.5);
      \draw[-, dotted, thick] (2.5,-0.5)--(1.5,-0.5);
      \draw[-, thick] (1.5,-0.5)--(1,-0.5);

      \fill (1,-0.5) circle(2pt) node[below] {$v_{2k+2}$}; 
      
      \fill (0,-1) circle(2pt) node[below] {$v_{2k+3}$}; 
      \draw[-, thick] (0,0)--(0,-1);
      
   \end{tikzpicture}
   \caption{$\theta^{\prime}(1,2,2k)$} \label{Figura F(1,2,2k,1)}  
\end{figure}  
The distance matrix of $\theta^{\prime}(1,2,2k)$ is
$$
D(\theta^{\prime}(1,2,2k)) =  
\left(\begin{array}{ccc}
 0 & v^t         & 1\\
 v & D(C_{2k+1}) & v + \mathbf{1}\\
 1 & v^t + \mathbf{1}^t   & 0
\end{array}\right),
$$
where $D(C_{2k+1})$ is the distance matrix of the cycle induced by the vertices $v_2, \dots, v_{2k+2}$ and $v^t = 
(1,2,\dots,k,k+1,k,\dots,2,1)$. By \eqref{eq inv theta(1,2,2k)}, we have that 
$$
\left(\begin{array}{cc}
 0 & v^t         \\
 v & D(C_{2k+1}) 
\end{array}\right)^{-1} = M_1^t M_2 M_1.
$$
If we define 
\begin{eqnarray*}
M_3 &:=&  
\left(\begin{array}{cc}
 {I}  & 0 \\
 -w^t M_1^t M_2 M_1 & {1}
\end{array}\right),
\\[2ex]
M_4 &:=&  
\left(\begin{array}{cc}
 M_1^t M_2 M_1 & 0 \\
 0 & (-w^t M_1^t M_2 M_1 w)^{-1} 
\end{array}\right),
\end{eqnarray*}
with $w^t := (1, v^t + \mathbf{1}^t)$, then 
$$D(G)^{-1} = M_3^t M_4 M_3$$    
and
$$
\det D(G) = \det M_4^{-1} = - w^t M_1^t M_2 M_1 w \ \det
\left(\begin{array}{cc}
 0 & v^t         \\
 v & D(C_{2k+1}) 
\end{array}\right)
.
$$
Combining this result with \eqref{eq det theta(1,2,2k) final}, we conclude that 
$$
\det D(G) = w^t M_1^t M_2 M_1 w \ (k+1)^2.
$$
Now we will calculate $w^t M_1^t M_2 M_1 w$, by \eqref{eq inv C_{2k+1}} we obtain
{\footnotesize
\begin{eqnarray*}
w^t M_1^t M_2 M_1 w  & = & (0,v^t) M_1^t M_2 M_1 \left(\begin{array}{c} 0\\ v \end{array}\right) + 2 \ (0,v^t) M_1^t M_2 M_1 \left(\begin{array}{c} 1\\ 
                           \mathbf{1} \end{array}\right)\\ 
                     &   & + (1,\mathbf{1}^t) M_1^t M_2 M_1 \left(\begin{array}{c} 1\\ \mathbf{1} \end{array}\right)\\
                     & = & (0,v^t)  \left(\begin{array}{c} 1\\ \mathbf{0} \end{array}\right) + 2 \ (1, \mathbf{0}^t) \left(\begin{array}{c} 1\\ 
                           \mathbf{1} \end{array}\right) + (1,\mathbf{1}^t) M_1^t M_2 M_1 \left(\begin{array}{c} 1\\ \mathbf{1} \end{array}\right)\\  
                     & = & 2 + (1,\mathbf{1}^t) M_1^t M_2 M_1 \left(\begin{array}{c} 1\\ \mathbf{1} \end{array}\right)\\
                     & = & 2 + \frac{v^t D(C_{2k+1})^{-1}v \ \mathbf{1}^t D(C_{2k+1})^{-1}\mathbf{1} - \Big(v^t D(C_{2k+1})^{-1}\mathbf{1} - 1\Big)^2}{v^t 
                           D(C_{2k+1})^{-1}v}.                           
\end{eqnarray*}
}
By \eqref{eq inv C_{2k+1}} we have that
{\footnotesize
\begin{eqnarray*}
v^t D(C_{2k+1})^{-1} \mathbf{1} &=& -2 v^t\mathbf{1} - v^t C^k \mathbf{1} - v^t C^{k+1} \mathbf{1} + \frac{2k+1}{k(k+1)} v^t J \mathbf{1} \\
                                &=& -8 \sum_{i=1}^k i - 4 (k+1) + \frac{2k+1}{k(k+1)} (2k+1) (2 \sum_{i=1}^k i + (k+1))\\
                                &=& -4 (k+1)^2 + \frac{(2k+1)^2}{k} (k+1) = \frac{k+1}{k},
\end{eqnarray*}
}
and
{\footnotesize
\begin{eqnarray*}
\mathbf{1}^t D(C_{2k+1})^{-1}\mathbf{1} &=& -2 \mathbf{1}^t\mathbf{1} - \mathbf{1}^t C^k \mathbf{1} - \mathbf{1}^t C^{k+1} \mathbf{1} + \frac{2k+1}{k(k+1)} 
                                             \mathbf{1}^t J \mathbf{1} \\
                                        &=& -4 (2k+1) + \frac{2k+1}{k(k+1)} (2k+1)^2\\
                                        &=& \frac{2k+1}{k (k+1)}
\end{eqnarray*}
}
Thus, by \eqref{eq vt C(2k+1)^{-1} v}, we deduce that
{\footnotesize
$$ 
w^t M_1^t M_2 M_1 w  =  2 + \frac{\frac{k+1}{k} \ \frac{2k+1}{k (k+1)} - \Big(\frac{k+1}{k} - 1\Big)^2}{\frac{k+1}{k}}= \frac{2k+4}{k+1}.
$$}
Finally, we obtain 
\begin{equation}\label{eq det F(1,2,2k) final}
\det D(\theta^{\prime}(1,2,2k)) = (2k+4)(k+1) = -n(n+2m) (-2)^{m-2},
\end{equation}
with $n=p+q$ and $m=1$, {where $p=2$ and $q=2s$}.

\textbf{Case 2:} 

Let   $\widehat{H}=\theta^{\prime}(1,2s,2k)$ and $\widehat{G}=\theta^{\prime}(1,2(s-1),2(k+1),1)$ be the graphs with its vertices labeled as in Figure~\ref{Figura F(1,2s,2k,1)} and Figure~\ref{Figura 
F(1,2(s-1),2(k+1),1)}, respectively, {for some $k\geq 2$ and $s\geq 2$}.

\begin{figure}[H]
  \centering
  \begin{tikzpicture}[scale=1]
      
      \fill (-3,0) circle(2pt) node[left] {$v_{1}$};
      \draw[-, thick] (-3,0)--(-2,-0.5); 
      \draw[-, thick] (-3,0)--(-2,0.5); 
      
      \fill (-2,0.5) circle(2pt) node[above] {$v_{2}$}; 
      \draw[-, thick] (-2,0.5)--(-1.5,0.5);
      \draw[-, dotted, thick] (-1.5,0.5)--(-0.5,0.5);
      \draw[-, thick] (-0.5,0.5)--(0,0.5);
      
      \draw[-, thick] (0,-0.5)--(0,0.5); 
      
      \fill (0,0.5) circle(2pt) node[above] {$v_{s+1}$}; 
      \draw[-, thick] (0,0.5)--(0.5,0.5);
      \draw[-, dotted, thick] (0.5,0.5)--(1.5,0.5);
      \draw[-, thick] (1.5,0.5)--(2,0.5);

      \fill (2,0.5) circle(2pt) node[above] {$v_{s+k}$};

      \fill (3,0) circle(2pt) node[right] {$v_{s+k+1}$};
      \draw[-, thick] (3,0)--(2,-0.5); 
      \draw[-, thick] (3,0)--(2,0.5); 
           
      \fill (2,-0.5) circle(2pt) node[below] {$v_{s+k+2}$};
      \draw[-, thick] (2,-0.5)--(1.5,-0.5);
      \draw[-, dotted, thick] (1.5,-0.5)--(0.5,-0.5);
      \draw[-, thick] (0.5,-0.5)--(0,-0.5);

      \fill (0,-0.5) circle(2pt) node[below] {$v_{s+2k+1}$}; 
      
      \fill (-2,-0.5) circle(2pt) node[below] {$v_{2k+2s}$}; 
      \draw[-, thick] (-2,-0.5)--(-1.5,-0.5);
      \draw[-, dotted, thick] (-1.5,-0.5)--(-0.5,-0.5);
      \draw[-, thick] (-0.5,-0.5)--(0,-0.5);
      
      \fill (-3,-1) circle(2pt) node[below] {$v_{2s+2k+1}$}; 
      \draw[-, thick] (-3,0)--(-3,-1);
      
   \end{tikzpicture}
   \caption{$\theta^{\prime}(1,2s,2k)$} \label{Figura F(1,2s,2k,1)}  
\end{figure}
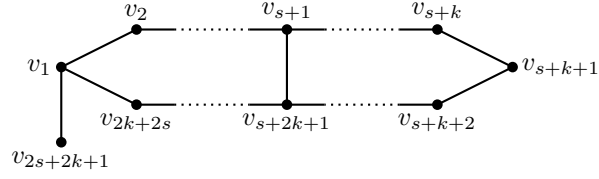

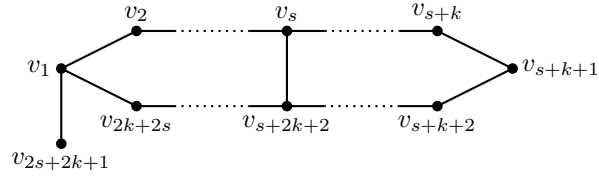
\begin{figure}[H]
  \centering
  \begin{tikzpicture}[scale=1]
      
      \fill (-3,0) circle(2pt) node[left] {$v_{1}$};
      \draw[-, thick] (-3,0)--(-2,-0.5); 
      \draw[-, thick] (-3,0)--(-2,0.5); 
      
      \fill (-2,0.5) circle(2pt) node[above] {$v_{2}$}; 
      \draw[-, thick] (-2,0.5)--(-1.5,0.5);
      \draw[-, dotted, thick] (-1.5,0.5)--(-0.5,0.5);
      \draw[-, thick] (-0.5,0.5)--(0,0.5);
      
      \draw[-, thick] (0,-0.5)--(0,0.5); 
      
      \fill (0,0.5) circle(2pt) node[above] {$v_{s}$}; 
      \draw[-, thick] (0,0.5)--(0.5,0.5);
      \draw[-, dotted, thick] (0.5,0.5)--(1.5,0.5);
      \draw[-, thick] (1.5,0.5)--(2,0.5);

      \fill (2,0.5) circle(2pt) node[above] {$v_{s+k}$};

      \fill (3,0) circle(2pt) node[right] {$v_{s+k+1}$};
      \draw[-, thick] (3,0)--(2,-0.5); 
      \draw[-, thick] (3,0)--(2,0.5); 
           
      \fill (2,-0.5) circle(2pt) node[below] {$v_{s+k+2}$};
      \draw[-, thick] (2,-0.5)--(1.5,-0.5);
      \draw[-, dotted, thick] (1.5,-0.5)--(0.5,-0.5);
      \draw[-, thick] (0.5,-0.5)--(0,-0.5);

      \fill (0,-0.5) circle(2pt) node[below] {$v_{s+2k+2}$}; 
      
      \fill (-2,-0.5) circle(2pt) node[below] {$v_{2k+2s}$}; 
      \draw[-, thick] (-2,-0.5)--(-1.5,-0.5);
      \draw[-, dotted, thick] (-1.5,-0.5)--(-0.5,-0.5);
      \draw[-, thick] (-0.5,-0.5)--(0,-0.5);
      
      \fill (-3,-1) circle(2pt) node[below] {$v_{2s+2k+1}$}; 
      \draw[-, thick] (-3,0)--(-3,-1);
      
   \end{tikzpicture}
   \caption{$\theta^{\prime}(1,2(s-1),2(k+1))$} \label{Figura F(1,2(s-1),2(k+1),1)}  
\end{figure}

The distance matrices of $\widehat G$ and $\widehat H$ are
$$
D(\widehat G) = 
\left(
\begin{array}{ccc}
    P        &     A^t            & v + \mathbf{1} \\
      A            &   P          & w_1 + \mathbf{1}\\
v^t + \mathbf{1}^t & w_1^t + \mathbf{1}^t & 0
\end{array}
\right),
$$
$$
D(\widehat H) = 
\left(
\begin{array}{ccc}
P & B^t  & v + \mathbf{1} \\
B & P & w_2 + \mathbf{1}\\
v^t + \mathbf{1}^t & w_2^t + \mathbf{1}^t & 0
\end{array}
\right),
$$
where $P$, $A$ and $B$ are the matrices defined in  \ref{fla1}, \ref{fla2} and \ref{fla3}, respectively, $v$ is the first column of $P$, $w_1$ is the first column of $A$ and 
$w_2$ is the first column of $B$.  

We claim that 
\begin{equation} \label{eq D(hatG) = N D(hatH) Nt} 
D(\widehat G) = 
\left(
\begin{array}{cc}
    N       &     0           \\
    0  &   1         
\end{array}
\right)
D(\widehat H)
\left(
\begin{array}{cc}
    N^t       &     0           \\
    0  &   1         
\end{array}
\right),
\end{equation}
where
$$
N = 
\left(
\begin{array}{cc}
    I        &     0           \\
    A P^{-1} - M B P^{-1}  &   M         
\end{array}
\right).
$$

Indeed, by \eqref{eq D(G) = N D(H) Nt}, we have
$$
\left(
\begin{array}{cc}
    P        &     A^t    \\
    A        &   P 
\end{array}
\right) = N
\left(
\begin{array}{ccc}
P & B^t \\
B & P 
\end{array}
\right)
N^t.
$$
Hence, it is sufficient to prove that
$$
N
\left(
\begin{array}{c}
 v + \mathbf{1} \\
 w_2 + \mathbf{1}
\end{array}
\right)
= 
\left(
\begin{array}{c}
 v + \mathbf{1} \\
 w_1 + \mathbf{1}
\end{array}
\right).
$$
It is easy to check that 
$$
N
\left(
\begin{array}{c}
 v + \mathbf{1} \\
 w_2 + \mathbf{1}
\end{array}
\right)
= 
\left(
\begin{array}{c}
 v + \mathbf{1} \\
 (A - M B) P^{-1} (v + \mathbf{1}) + M(w_2 + \mathbf{1})
\end{array}
\right).
$$
Since $v$ is the first column of $P$, we obtain 
\begin{eqnarray*}
 (A - M B) P^{-1} v + Mw_2 & = & (A - M B) e_1 + Mw_2\\
                          & = & w_1 - Mw_2 + Mw_2  =  w_1.
\end{eqnarray*}
From the proof of Theorem \ref{theta}, Case 2, we have that
$$
(A - M B)P^{-1} =  e_2 e_1^t + \sum_{i=3}^{k+s} e_{i} e_{s}^t - \sum_{i=2}^{k+s} e_{i} e_{s+1}^t,
$$
combining this with the definition of $M$, we see that
$$
 (A - M B) P^{-1} \mathbf{1} +  M \mathbf{1} =  0 + \mathbf{1} = \mathbf{1}.
$$
This {completes} the proof of \eqref{eq D(hatG) = N D(hatH) Nt}. 
Futhermore, since 
$$
\det\left(
\begin{array}{cc}
    N       &     0           \\
    0  &   1         
\end{array}
\right)\cdot
\det
\left(
\begin{array}{cc}
    N^t       &     0           \\
    0  &   1         
\end{array}
\right)
= 1, 
$$
we deduce that 
$$
\det D(\widehat G) =  \det D(\widehat H).
$$
Combining this with \eqref{eq det F(1,2,2k) final}, using an inductive argument, we have 
$$
\det {D}(\theta^{\prime}(1,2s,2k)) = -n(n+2m)(-2)^{m-2},
$$
with $n=p+q$ y $m=1$, where $p=2s$ and $q=2k$.
\end{proof}

We now consider the case when a path is attached to a vertex of $\theta (l,p,q).$ We denote by $\theta^{\prime}_{m}(l,p,q)$ the graph obtained from $\theta(l,p,q)$  by identifying one vertex of degree  three of $\theta(l,p,q)$  with one vertex of degree one of a path of length $m \geq 0.$

The next proposition investigates the determinant of these graphs, when $p$ and $q$ are even.
 
\begin{prop} If $p$ and $q$ are even integers, then
$$\det D(\theta^{\prime}_{m}(1,p,q)) = -n (n + 2m) (-2)^{m-2},$$ 
where $n = p+q$ and $m \geq 0$.

\end{prop}

\begin{proof}

Let  $G_{m}=\theta^{\prime}_{m}(1,p,q),$ 
$V(G_m)= \{1, \dots p+q, \dots, 
p+q+m\}$ such that the vertices $\{1, \dots, p+q\}$ induce $\theta(1,2s,2k)$ and the vertices $\{p+q, \dots, p+q+m\}$ induce $P_{m+1}$, where $p=2s$, $q=2k$ and 
$m \ge 0$ for some $k \ge 1$ and $s \ge {1}$. Arguing as in \cite[Theorem 3.2]{GoZhXu2013}, we obtain
$$
\det D(G_m) = -4 \det D(G_{m-1})  - 4\det D(G_{m-2}),
$$
for $m \ge 2$. Combining this identity with the results of Theorem \ref{theta}, case $(a)$ and Theorem \ref{theta prime}, case $(a)$,  we deduce that
 $$\det D\big(G_m\big) = -n (n + 2m) (-2)^{m-2},$$ 
where $n = p+q$ and $m \ge 0$.
\end{proof}

 As we already know  the determinant of a $\theta$-graph and $\theta^{\prime}$-graph, we  obtain the {values of $\cof D(G)$}, for $G= \theta({l,p,q}).$

\begin{cor}
The following assertions hold:
\begin{itemize}
\item If $G = \theta(1, p, q)$ for some even integers $p$ and $q,$ then $\cof D(G) = -(p + q).$
\item If $G = \theta(2, 2, 2),$ then
$\cof D(G) = -16.$
\item If $G = \theta(2, 2, q)$ for some odd integer $q > 1,$ then
$\cof D(G) = 4q - 8.$
\item Otherwise, $\cof D(G) = 0.$
\end{itemize}
\end{cor}

\begin{rem} A graph is said to be \emph{at most bicyclic} if it arises from a tree by the addition of at most two edges. The blocks that are at most bicyclic graphs having at least two vertices are: edge blocks,  cycles, and  $\theta$-graphs. The values of $\det D(G)$ and $\cof D(G)$ where already known in the first two cases. Now, we have {obtained} the values of $\det D(G)$ and $\cof D(G)$ for the last case.  
\end{rem}

From the results above, by applying Theorem \ref{2}, we present in the sequence a formula for $\det D(G)$ for all graphs having at most bicylic blocks. Notice that this class generalizes the class of cacti (which are graphs having at most unicyclic blocks).

\begin{teo}Let G be a connected graph having at most bicyclic blocks. If $G = K_1$ or any block of $G$ is {an even} cycle or a $\theta({l,p,q})$  with $\det D(\theta({l,p,q})) = 0,$
then $\det D(G) = \cof D(G) = 0.$ Otherwise, if the blocks of $G$ are:
\begin{itemize}
\item $m$ edge blocks,
\item $c$ odd cycles of lengths $l_1 , l_2 ,\hdots , l_c,$
\item $r$ graphs $\theta(1, p_1, q_1), \theta(1, p_2 , q_2 ), \hdots, \theta(1, p_r , q_r )$ for even integers $p_1, q_1, \hdots, p_r, q_r,$
\item $s$ graphs $\theta(2, 2, 2),$ and
\item $t$ graphs $\theta(2, 2, q_1 ), \theta(2, 2, q_2 ), \hdots , \theta(2, 2, q_t )$ for odd
integers $q_1, q_2, \hdots, q_t>1$,
\end{itemize}
then
$$\det D(G)=\left( \frac{m}{2} +\sum\limits_{h=1}^c\frac{l_h^2-1}{4l_h}+\sum_{i=1}^r\frac{p_i+q_i}{4}+s+\sum\limits_{j=1}^t\frac{q_j^2-5}{4q_j-8} \right) \cof D(G),$$
where
$$\cof D(G)=(-2)^m(-1)^r(-16)^s\left(\prod\limits_{h=1}^cl_h\right)\left(\prod\limits_{i=1}^r(p_i+q_i)\right)\prod\limits_{j=1}^t(4q_j-8).$$

\end{teo}






\end{document}